\documentclass[10pt]{amsart}
\usepackage{amssymb,amsmath,amscd}
\usepackage[T1]{fontenc}

\usepackage{mathrsfs}
\usepackage{color}
\usepackage{epsfig}

\usepackage[all]{xy}

\newtheorem{theorem}{Theorem}

\newtheorem{corollary}[theorem]{Corollary}

\newcommand{\Z}{\mathbb{Z}}

\begin{document}

\title[Divisors, Chebyshev, and Quasi-modular Forms]{MacMahon's sum-of-divisors functions, Chebyshev polynomials, and Quasi-modular forms}
\author{George E. Andrews}
\address{Penn State University, USA}
\email{andrews@math.psu.edu}

\author{Simon C. F. Rose}
\address{University of British Columbia, Canada}
\email{scfr@math.ubc.ca}

\begin{abstract}
We investigate a relationship between MacMahon's generalized sum-of-divisors functions and Chebyshev polynomials of the first kind. This determines a recurrence relation to compute these functions, as well as proving a conjecture of MacMahon about their general form by relating them to quasi-modular forms. These functions arise as solutions to a curve-counting problem on Abelian surfaces.
\end{abstract}

\maketitle

\section{Introduction}

The sum-of-divisors function $\sigma_k(n)$ is defined to be
\[
\sigma_k(n) = \sum_{d \mid n} d^k.
\]
For $k = 1$, this has as a generating function
\[
A_1(q) = \sum_{k=1}^\infty \sigma_1(n)q^n = \sum_{k=1}^\infty \frac{q^k}{(1 - q^k)^2}.
\]
As a generalization of this notion, MacMahon introduces in the paper \cite[pp. 303, 309]{MacMahon} the generating functions
\begin{gather*}
A_k = \sum_{0<m_1 < \cdots < m_k}\frac{q^{m_1 + \cdots + m_k}}{(1 - q^{m_1})^2 \cdots (1 - q^{m_k})^2} \\
C_k = \sum_{0<m_1 < \cdots < m_k}\frac{q^{2m_1 + \cdots +2m_k - k}}{(1 - q^{2m_1 - 1})^2 \cdots (1 - q^{2m_k - 1})^2}.
\end{gather*}
These provide generalizations in the following sense.

Fix a positive integer $k$. We define $a_{n,k}$ to be the sum
\[
a_{n,k} = \sum s_1 \cdots s_k
\]
where the sum is taken over all possible ways of writing $n = s_1m_1 + \cdots + s_km_k$ with $0 < m_1 < \cdots < m_k$. Note that for $k = 1$ this is nothing but $\sigma_1(n)$, the usual sum-of-divisors function. It can then be shown that we have
\[
A_k(q) = \sum_{n = 1}^\infty a_{n,k}q^n.
\]

Similarly, we define $c_{n,k}$ to be
\[
c_{n,k} = \sum s_1 \cdots s_k
\]
where the sum is over all partitions of $n$ into
\[
n = s_1(2m_1 - 1) + \cdots + s_k(2m_k - 1)
\]
with, as before $0 < m_1 < \cdots < m_k$. For $k = 1$ this is the sum over all divisors whose conjugate is an odd number. As for the case of $a_{n,k}$, we have
\[
C_k(q) = \sum_{n=1}^\infty c_{n,k}q^n.
\]

We recall also that Chebyshev polynomials \cite[p. 101]{aar} are defined via the relation
\[
T_n(\cos \theta) = \cos(n\theta).
\]
With these we form the following generating functions.
\begin{gather*}
F(x,q) := 2\sum_{n=0}^\infty T_{2n+1}(\tfrac{1}{2}x)q^{n^2+n} \\
G(x,q) := 1+2\sum_{n=1}^\infty T_{2n}(\tfrac{1}{2}x)q^{n^2}.
\end{gather*}

The results of this paper are the following.

\begin{theorem}\label{thm_main}
We have the following equalities:
\begin{gather*}
F(x,q)  = (q^2;q^2)_\infty^3 \sum_{k=0}^\infty A_k(q^2)x^{2k+1}\\
G(x,q) = \frac{(q;q)_\infty}{(-q;q)_\infty}\sum_{k=0}^\infty C_k(q)x^{2k}
\end{gather*}
where $(a;q)_\infty = \prod_{k=0}^\infty (1 - aq^k)$.
\end{theorem}

\begin{corollary}\label{explicit_form}
The functions $A_k(q)$ and $C_k(q)$ can be written as
\begin{gather*}
A_k(q) = \frac{(-1)^k}{(2k+1)!(q;q)_\infty^3}\sum_{n=0}^\infty (-1)^n(2n+1)\frac{(n+k)!}{(n-k)!}q^{\tfrac{1}{2}n(n+1)} \\
C_k(q) = \frac{(-1)^k(-q;q)^\infty}{(2k)!(q;q)_\infty}\sum_{n=1}^\infty (-1)^n 2n \frac{(n+k-1)!}{(n-k)!}q^{n^2}.
\end{gather*}
\end{corollary}

\begin{corollary}\label{divisor_recurrence}
The functions $A_k$ and $C_k$ satisfy the recurrence relations
\begin{gather*}
A_k(q) = \frac{1}{(2k+1)2k}\Big(\big(6A_1(q) + k(k-1)\big)A_{k-1}(q) - 2q\frac{d}{dq}A_{k-1}(q)\Big) \\
C_k(q) = \frac{1}{2k(2k-1)}\Big(\big(2C_1(q) + (k-1)^2\big)C_{k-1}(q) - q\frac{d}{dq}C_{k-1}(q)\Big).
\end{gather*}
\end{corollary}

Our final result settles a long-standing conjecture of MacMahon. In MacMahon's paper \cite[p. 328]{MacMahon} he makes the claim
\begin{quote}
The function $A_k = \sum a_{n,k}q^n$ has apparently the property that the coefficient $a_{n,k}$ is expressible as a linear function of the sum of the uneven powers of the divisors of $n$. I have not succeeded in reaching the general theory...
\end{quote}
What we prove is the following.

\begin{corollary}
The functions $A_k(q)$ are in the ring of quasi-modular forms.
\end{corollary}

We will also discuss in section \ref{section_applications} some applications of this result to an enumerative problem involving counting curves on abelian surfaces.

\section{Proofs}

\begin{proof}[Proof of theorem \ref{thm_main}]
Beginning with the series $F(x,q)$, and letting $x = 2\cos\theta$ we find
\begin{align*}
F(x,q) &= 2\sum_{n=0}^\infty T_{2n+1}(\cos\theta)q^{n^2+n}\\
 &= 2\sum_{n=0}^\infty \cos\big((2n+1)\theta\big)q^{n^2+n}\\
 &= \sum_{n=0}^\infty \Big(e^{i(2n+1)\theta} + e^{-i(2n+1)\theta}\Big)q^{n^2+n}\\
 &= \sum_{n=0}^\infty e^{i(2n+1)\theta}e^{n^2+n} + \sum_{n=0}^\infty e^{-i{2n+1}\theta}q^{n^2+n}
\end{align*}
where in the latter sum, letting $n \mapsto -n - 1$ we obtain
\begin{equation*}
F(x,q) = e^{i\theta}\sum_{n=-\infty}^\infty e^{2ni\theta}q^{n^2+n}.
\end{equation*}
Using the Jacobi triple product \cite[p. 497, Thm 10.4.1]{aar} we see that this is equal to
\begin{align*}
e^{i\theta}\sum_{n=-\infty}^\infty e^{2ni\theta}q^{n^2+n} &= e^{i\theta}(-e^{-2i\theta};q^2)_\infty(-q^2e^{2i\theta};q^2)_\infty(q^2;q^2)_\infty \\
 &= \underbrace{(e^{i\theta} + e^{-i\theta})}_x(q^2;q^2)_\infty \prod_{m=1}^\infty\big(q + \underbrace{2\cos(2\theta)}_{x^2 - 2}q^{2m} + q^{4m}\Big) \\
 &= x(q^2;q^2)_\infty \prod_{m=1}^\infty \big((1 - q^{2m})^2 + x^2q^{2m}\big) \\
 &= x(q^2;q^2)_\infty^3\prod_{m=1}^\infty \Big(1 + x^2\frac{q^{2m}}{(1 - q^{2m})^2}\Big) \\
 &= (q^2;q^2)_\infty^3 \sum_{k=0}^\infty A_k(q^2)x^{2k+1}
\end{align*}
and thus comparing coefficients of $x^{2k+1}$ yeilds the result.

We ply a similar trick for $G(x,k)$. In that case we have
\begin{align*}
G(x,q) &= 1 + 2\sum_{n > 0} T_{2n}(\cos\theta)q^{n^2} \\
 &= 1 + 2\sum_{n>0}\cos(2n\theta)q^{n^2} \\
 &= \sum_{n=-\infty}^\infty e^{2ni\theta}q^{n^2}
\end{align*}
which, again, by the Jacobi triple product yields
\begin{align*}
\sum_{n=-\infty}^\infty e^{2ni\theta}q^{n^2} &= (q^2;q^2)_\infty (-qe^{2i\theta};q^2)_\infty(-qe^{-2i\theta};q^2)_\infty \\
 &= (q^2;q^2)_\infty \prod_{m=1}^\infty \big(1 + \underbrace{2\cos(2\theta)}_{x^2 - 2}q^{2m-1} + q^{4m - 2}\big) \\
 &= (q^2;q^2)_\infty \prod_{m=1}^\infty \big((1 - q^{2m-1})^2 + x^2q^{2n-1}\big) \\
 &= \underbrace{(q^2;q^2)_\infty(q;q^2)_\infty^2}_{\frac{(q;q)_\infty}{(-q;q)_\infty}}\prod_{m=1}^\infty\Big(1 + x^2\frac{q^{2m-1}}{(1 - q^{2m-1})^2}\Big) \\
 &= \frac{(q;q)_\infty}{(-q;q)_\infty} \sum_{k=0}^\infty C_k(q)x^{2k}
\end{align*}
which completes the theorem. 
\end{proof}

To deduce Corollary \ref{explicit_form}, we begin by expanding the series $F(x,q)$ (and similarly, $G(x,q)$) in powers of $x$, i.e.
\begin{gather*}
F(x,q) = xf_0(q) + x^3f_1(q) + x^5f_2(q) + \cdots + x^{2k+1}f_k(q) + \cdots \\
G(x,q) = g_0(q) + x^2 g_1(q) + x^4 g_2(q) + \cdots + x^{2k}g_k(q) + \cdots.
\end{gather*}

Now, it can be shown that the coefficients of $x^{2k}$ in $2T_{2n}(\tfrac{1}{2}x)$ and of $x^{2k+1}$ in $2T_{2n+1}(\tfrac{1}{2}x)$ are respectively given by
\[
2n(-1)^{n-k}\frac{(n+k-1)!}{(n-k)!(2k)!} \qquad \qquad (-1)^{n-k}(2n+1)\frac{(n+k)!}{(n-k)!(2k+1)!}
\]
and thus we have
\begin{gather*}
f_k(q) = \frac{(-1)^k}{(2k+1)!}\sum_{n=0}^\infty (-1)^n(2n+1)\frac{(n+k)!}{(n-k)!}q^{n^2+n} \\
g_k(q) = \frac{(-1)^k}{(2k)!}2\sum_{n=1}^\infty (-1)^n n \frac{(n + k - 1)!}{(n-k)!}q^{n^2}.
\end{gather*}
As theorem \ref{thm_main} implies that $f_k(q) = (q^2;q^2)_\infty^3 A_k(q^2)$ and $g_k(q) = \frac{(q;q)_\infty}{(-q;q)_\infty}C_k(q)$, we see that Corollary \ref{explicit_form} follows.

Next, letting
\begin{gather*}
f_0(q) = \sum_{n=0}^\infty (-1)^n(2n+1)q^{n^2+n} = (q^2;q^2)_\infty^3\\
g_0(q) = 1 + 2\sum_{n=1}^\infty (-1)^n q^{n^2} = \frac{(q;q)_\infty}{(-q;q)_\infty}
\end{gather*}
and defining the operators $D_\ell = q\frac{d}{dq} - \ell(\ell - 1)$ and $D'_\ell = q\frac{d}{dq} - (\ell - 1)^2$, we then have that
\begin{gather*}
f_k(q) = \frac{(-1)^k}{(2k+1)!}D_k \cdots D_1 f_0(q) \\
g_k(q) = \frac{(-1)^k}{(2k)!}D'_k \cdots D'_1 g_0(q).
\end{gather*}

From these formulae we note that the functions $f_k$, $g_k$ satisfy the recursion relations
\begin{gather*}
f_k(q) = \frac{-1}{(2k+1)2k}\Big(q\frac{d}{dq} - k(k-1)\Big)f_{k-1}(q) \\
g_k(q) = \frac{-1}{2k(2k-1)}\Big(q\frac{d}{dq} - (k-1)^2\Big)g_{k-1}(q).
\end{gather*}

Noting again that $f_k(q) = (q^2;q^2)_\infty^3 A_k(q^2)$ and $g_k(q) = \frac{(q;q)_\infty}{(-q;q)_\infty}C_k(q)$, we now obtain the recurrence relation of Corollary \ref{divisor_recurrence} between the functions $A_k(q)$ and $C_k(q)$.

Our final result requires a bit of explanation. It is well known that the ring of modular forms for the full modular group $\Gamma = PSL_2(\Z)$ is the polynomial ring in the generators $E_4$, $E_6$, where
\[
E_{2k}(q) = 1 + \frac{2}{\zeta(1 - 2k)}\sum_{n=1}^\infty \sigma_{2k-1}(n)q^n
\]
are the Eisenstein series of weights $2k$. There are no modular forms of weight 2 for $\Gamma$, but $E_2 = 1 -24\sum_{n=1}^\infty \sigma_1(n)q^n$ is a {\em quasi}-modular form (See \cite{quasimodular}).

The relevant fact for this paper is that the ring of all such objects (which contains the ring of modular forms as a subring) is the ring generated either by $E_2, E_4$, and $E_6$, or by $q\frac{d}{dq}$ and by $E_2$. Noting then that $A_1(q) = \frac{1 - E_2(q)}{24}$, the recurrence relation from Corollary \ref{divisor_recurrence} implies that each $A_k(q)$ lies in this ring, and hence the conclusion follows.

\section{Applications}\label{section_applications}

The functions $A_k(q)$ and $C_k(q)$ arise naturally in the following problem in enumerative algebraic geometry.

Let $A \subset \mathbb{P}^N$ be a generic polarized abelian surface. There are then a finite number of hyperplane sections which are hyperelliptic curves of geometric genus $g$ and having $\delta = N - g + 2$ nodes. The number of such curves, $N(g, \delta)$ is independent of the choice of $A$ and these numbers can be assembled into a generating function
\[
F(x,u) = \sum_{g, \delta} N(g,\delta) x^g u^\delta.
\]
The coefficient of $x^g$ in $F$ is given by a certain homogeneous polynomial of degree $g-1$ in the functions $A_k(u^4)$ and $C_k(u^2)$.

This formula is derived by relating hyperelliptic curves on $A$ to genus zero curves on the Kummer surface $A/\pm1$. The latter is computed using orbifold Gromov-Witten theory, the Crepant resolution conjecute \cite{crc} and the Yau-Zaslow formula \cite{yz} \cite{bl}. This will be described further in the second author's thesis.

\bibliographystyle{amsplain}
\bibliography{cheb}

\end{document}